\documentclass[12pt,a4paper]{amsart}
\allowdisplaybreaks[1]

\usepackage{amstext}
\usepackage{amsthm}
\usepackage{amssymb}

\usepackage{tikz}
\usepackage[all]{xy}
\usepackage{graphics}
\usepackage[top=30truemm,bottom=30truemm,left=25truemm,right=25truemm]{geometry}

\DeclareFontEncoding{OT2}{}{}
\DeclareFontSubstitution{OT2}{cmr}{n}{l}

\DeclareFontFamily{OT2}{cmr}{\hyphenchar\font45}
\DeclareFontShape{OT2}{cmr}{n}{l}{%
<5><6><7><8><9>gen*wncyr%
<10><10.95><12><14.4><17.28><20.74><24.88>wncyr10}{}

\DeclareMathAlphabet{\mathcyr}{OT2}{cmr}{n}{l}

\newtheorem{thm}{Theorem}[section]
\newtheorem*{thm*}{Theorem}
\newtheorem{lem}[thm]{Lemma}

\theoremstyle{definition}
\newtheorem{defn}[thm]{Definition}
\newtheorem{ex}[thm]{Example}
\theoremstyle{remark}
\newtheorem{rem}[thm]{Remark}

\usepackage[T1]{fontenc}
\usepackage[latin9]{inputenc}
\usepackage{mathtools}
\usepackage{amsthm}
\usepackage{amssymb}

\begin{document}

\title{Multiple zeta-star values for indices of infinite length}
%Order structure of multiple zeta-star values}

\author{Minoru Hirose}
\address[Minoru Hirose]{Institute for Advanced Research, Nagoya University,  Furo-cho, Chikusa-ku, Nagoya, 464-8602, Japan}
\email{minoru.hirose@math.nagoya-u.ac.jp}

\author{Hideki Murahara}
\address[Hideki Murahara]{The University of Kitakyushu,  4-2-1 Kitagata, Kokuraminami-ku, Kitakyushu, Fukuoka, 802-8577, Japan}
\email{hmurahara@mathformula.page}

\author{Tomokazu Onozuka}
\address[Tomokazu Onozuka]{Institute of Mathematics for Industry, Kyushu University 744, Motooka, Nishi-ku, Fukuoka, 819-0395, Japan} \email{t-onozuka@imi.kyushu-u.ac.jp}


\begin{abstract}
In this paper, we consider infinite-length versions of multiple zeta-star values. 
We give several explicit formulas for the infinite-length versions of multiple zeta-star values. 
We also discuss the analytic properties of the map from indices to the infinite-length versions of multiple zeta-star values.
\end{abstract}

\maketitle

%%%%%%%%%%%%%%%%%%%%%%%%%%%%%%%%%%%%%%%%%%%%%%%%%%%%%%%%%%%%%%%%%%%%%%%%%%%%%%%%%%%%%%%%
\section{Main result}
The multiple zeta-star value is the convergent series 
\[
 \zeta^\star (k_1,\dots,k_r)
 =\sum_{n_1\ge \cdots \ge n_r\ge1} \frac{1}{n_1^{k_1}\cdots n_r^{k_r}}
\]
for $(k_1,\dots,k_r)\in\mathbb{Z}_{\ge1}^r$ with $k_1\ge2$ and has been studied variously, along with the multiple zeta values.
In this paper, we consider the infinite length version of the multiple zeta-star values.
\begin{defn}
For $(k_1,k_2,\dots)\in\mathbb{Z}_{\ge1}^{\infty}$ with $k_1\ge2$, we define multiple zeta-star values for indices of infinite length by
\begin{align*}
 \zeta^\star (k_1, k_2,\dots)
 &=\sum_{m_{1}\geq m_{2}\geq\cdots\geq1}\frac{1}{m_{1}^{k_{1}}m_{2}^{k_{2}}\cdots} \qquad( =\lim_{r\rightarrow\infty} \zeta^\star (k_1,\dots,k_r) ),
\end{align*} 
where the summation is over the all decreasing sequence $(m_{j})_{j=1}^{\infty}$ of positive integers
such that $\lim_{r\to\infty}m_{r}=1$.
\end{defn}
We will see later that the above sum converges except for the case where $k_1=2$ and $k_j=1$ for all $j>1$ (see Section \ref{order_property}).
First, we will show some formulas for the multiple zeta-star values for indices of infinite length.
Let $\{k\}^r$ denote the $r$ times repetition of the $k$, e.g., $(\{k\}^3)=(k,k,k)$.
% We sometimes abbreviate $\lim$, as when we write $\lim_{r\rightarrow\infty}(\{k\}^{r})$ as $\zeta^{\star}(\{k\}^{\infty})$, unless there is a risk of confusion.
We use the summation symbol $\sum'$ in an extended meaning of $\sum$, i.e.,
$\sum'{\!}_{j=a}^{b-1}$ means $-\sum_{j=b}^{a-1}$ if $b<a$, 0 if $b=a$, $\sum_{j=a}^{b-1}$ if $b>a$.
\begin{thm} \label{special_values}
We have the following equalities:
\begin{enumerate}
 \item For $k_1,\dots,k_r\in\mathbb{Z}_{\ge1}^r$ with $k_1\ge2$,
  \[
   \zeta^{\star}(k_{1},\dots,k_{r-1},k_{r}+1,\{1\}^{\infty})
   =\zeta^{\star}(k_{1},\dots,k_{r}).
  \]
 \item For $k_1,\dots,k_r\in\mathbb{Z}_{\ge1}^r$ with $k_1\ge2$,
 \begin{align*}
  &\zeta^{\star}(k_{1},\dots,k_{r},\{2\}^{\infty})\\
  &=(-1)^{k_{1}+\cdots+k_{r}}\left(2-2\sum_{s=1}^{r}\sideset{}{'}{\sum}_{j=2}^{k_{s}-1}
  (-1)^{k_{1}+\cdots+k_{s-1}+j}\zeta^{\star}(k_{1},\dots,k_{s-1},j)\right).
 \end{align*} 
 \item For $k\ge 2$,
  \[
  \zeta^{\star}(\{k\}^{\infty})=\prod_{m=2}^{\infty}\left(\frac{m^{k}}{m^{k}-1}\right) = \prod_{c^k=1}\Gamma(2-c).
  \]
  %=\sum_{s=0}^\infty\prod_{m=2}^\infty\frac{1}{\Gamma(m)^{ks}}todo$,
 \item For $n\ge 2$,
  \[
   \zeta^{\star}(\{2,\{1\}^{n-2}\}^{\infty})=n.
  \]
 \item For $n\ge 1$,
  \[
    \zeta^{\star}(\{\{2\}^{n},1\}^{\infty})=2\prod_{c^{2n+1}=1}\frac{\Gamma(2-c)}{\Gamma(2+c)}.
  \]  
 \item For $n\ge 0$,
  \begin{align*}
   \zeta^{\star}(\{\{2\}^{n},3,\{2\}^{n},1\}^{\infty})
   =2\prod_{s\in\{\pm1\}}\prod_{c^{2n+2}=s}\Gamma(2-c)^{-s}\Gamma\left(1-\frac{c}{2}\right)^{2s}.      
  \end{align*}
 \end{enumerate}
\end{thm}
\begin{ex} \label{ex:values}
We have the following equalities:
\begin{enumerate}
\item $\zeta^\star(\{4\}^{\infty})=\frac{8\pi}{e^{\pi}-e^{-\pi}},$
\item $\zeta^\star(3,\{2\}^{\infty})=2\zeta(2)-2,$
\item $\zeta^\star(\{3,1\}^{\infty})=\frac{4(e^{\pi}+1)}{\pi(e^{\pi}-1)},$
\item $\zeta^\star(\{2\}^{\infty})=2,$
\item $\zeta^\star(\{2,1\}^{\infty})=3$.
\end{enumerate}
\end{ex}

Second, let us define $Z^\star:[0,1]\to [1,\infty]$ by
$Z^\star(0)=1$ and
\[
Z^\star\left(\sum_{j=1}^{\infty}\frac{1}{2^{k_{1}+\cdots+k_{j}}}\right)
=\zeta^\star(k_{1}+1,k_{2},k_{3},\dots),
\footnote{
For example, we have
\[
Z^\star\left(\frac{1}{2}\right)
=Z^\star\left(\sum_{j=2}^{\infty}\frac{1}{2^{j}}\right)
=\zeta^\star(3,1,1,\dots)
=\zeta(2).
\]
}
\]
where $k_1,k_2,\dots\in\mathbb{Z}_{\ge1}$.
The function $Z^\star$ contains information of all multiple zeta-star values for indices of infinite length (see Figure \ref{fig:graph_F} for the graph of $Z^\star$).
\begin{figure}\label{fig:graph_F}
  \centering
  \includegraphics[width=15cm]{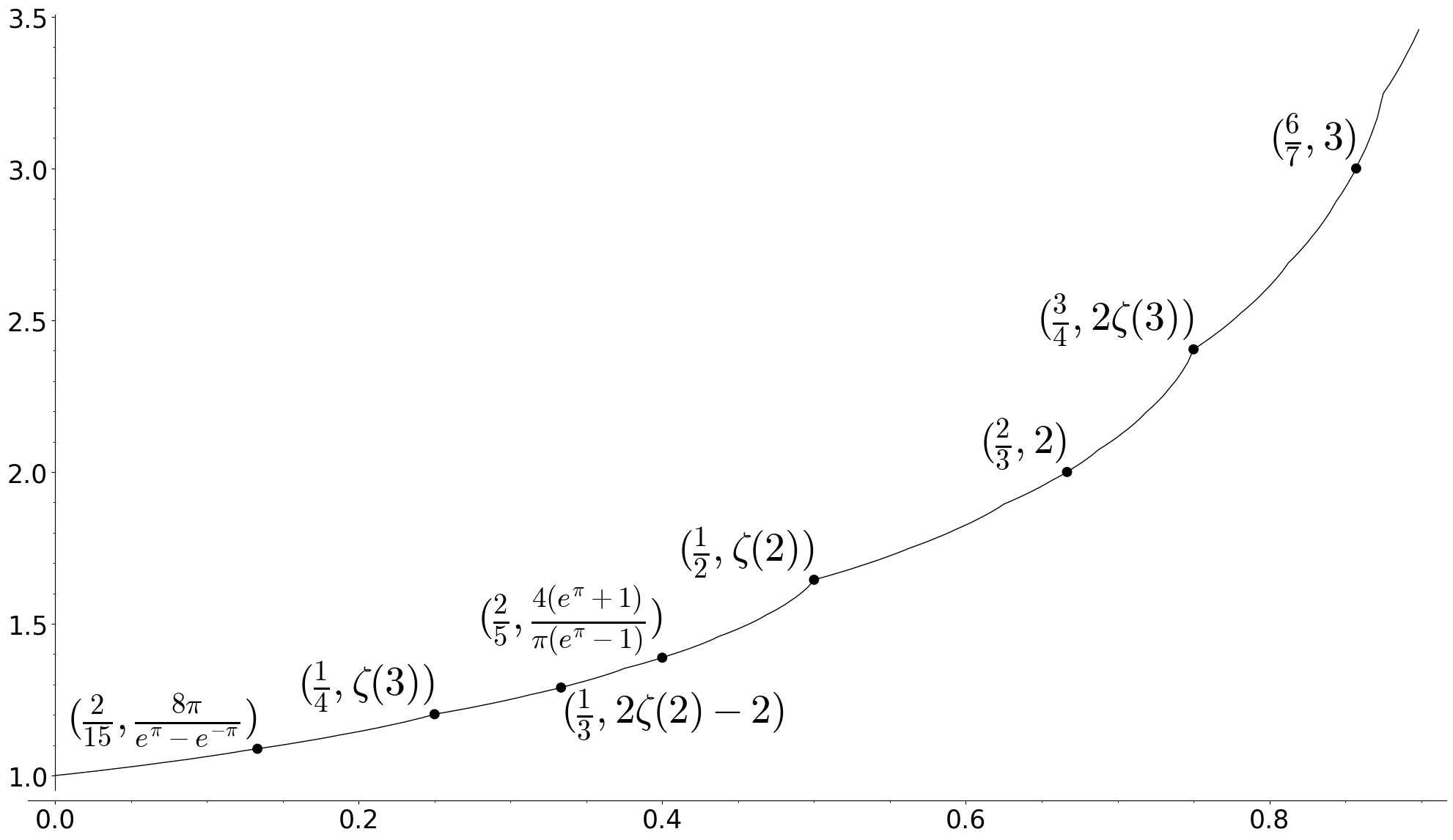}
  \caption{The graph and some special values of $Z^\star$.}
\end{figure}
Given two indices $(k_1,k_2,\dots)$ and $(l_1,l_2,\dots)$, we say that the former is lexicographically smaller than the latter if 
there exists $j$ such that $k_i=l_i$ and $k_j < l_j \; (i \in \{1,\dots,j-1\})$.
%, e.g., $(1,2)\prec (3)$.

\begin{thm} \label{main2}
$Z^\star$ is a continuous and bijective function, or equivalently,
the map 
\[
(k_{1},k_{2},k_{3},\dots)\mapsto\zeta^{\star}(k_{1}+1,k_{2},k_{3},\dots)
\]
 gives an order-reversing bijection between $(\mathbb{Z}_{\geq1}^{\infty},\prec)$
and $(1,\infty]$ where $\prec$ is the lexicographic order.
\end{thm}
\begin{rem}
The order structure for the set of multiple zeta values is studied by Kumar \cite{Kum}.
\end{rem}
\begin{rem}
 Li, independently of our study, obtained the same results as Theorem \ref{special_values} (1), (4), Theorem \ref{main2}, and further studied related topics. 
 For more details, see \cite{Jia}. 
\end{rem}

\begin{thm} \label{main3}
The map $Z^\star$ is not differentiable on some dense set.
More precisely, we have the followings:
\begin{enumerate}
 \item The map $Z^\star$ is right-differentiable at $z$ for $0\leq z<1$.
 \item The map $Z^\star$ is left-differentiable at $z$ if $z\not\in \{1-\frac{1}{2^n}\mid n>0\}$.
 \item The map $Z^\star$ is not left-differentiable at $z$ if $z\in \{1-\frac{1}{2^n}\mid n>0\}$.
 \item The left-differential $\partial_{-}Z^{\star}(z)$ is equal to the right-differential $\partial_{+}Z^{\star}(z)$ if $z\in(0,1)\setminus\left\{ \frac{a}{2^{n}}\mid0<a<2^{n},n>0\right\}$.
 \item The left-differential $\partial_{-}Z^{\star}(z)$ is greater than the right differential $\partial_{+}Z^{\star}(z)$ if $z\in\left\{ \frac{a}{2^{n}}\mid0<a<2^{n}-1,n>0\right\} $.
\end{enumerate}
\end{thm}
The proof of Theorem \ref{main3} will be given in Theorems \ref{47a}, \ref{48a}, \ref{cor:ccc}, and \ref{thm:ddd} (see also Remark \ref{rem:left=right}).

%%%%%%%%%%%%%%%%%%%%%%%%%%%%%%%%%%%%%%%%%%%%%%%%%%%%%%%%%%%%%%%%%%%%%%%%%%%%%%%%%%%%%%%%
\section{Special values}
\begin{lem}\label{lem:linf}
For $k_{1},\dots,k_{r},l\in\mathbb{Z}_{\geq1}$ with $k_{1}>1$, we
have
\[
\zeta^{\star}(k_{1},\dots,k_{r},\{l\}^{\infty})=\sum_{m_{1}\ge\cdots\ge m_{r}\ge1}\frac{1}{m_{1}^{k_{1}}\cdots m_{r}^{k_{r}}}\prod_{s=2}^{m_{r}}\frac{s^{l}}{s^{l}-1}.
\]
\end{lem}
\begin{proof}
It follows from the following calculation:
\begin{align*}
 & \zeta^{\star}(k_{1},\dots,k_{r},\{l\}^{\infty})\\
 & =\lim_{R\to\infty}\sum_{m_{1}\ge\cdots\ge m_{r}\ge n_{1}\ge\cdots\ge n_{R}\ge1}\frac{1}{m_{1}^{k_{1}}\cdots m_{r}^{k_{r}}n_{1}^{l}\cdots n_{R}^{l}}\\
 & =\sum_{m_{1}\ge\cdots\ge m_{r}\ge1}\frac{1}{m_{1}^{k_{1}}\cdots m_{r}^{k_{r}}}\lim_{R\to\infty}\sum_{m_{r}\ge n_{1}\ge\cdots\ge n_{R}\ge1}\frac{1}{n_{1}^{l}\cdots n_{R}^{l}}\\
 & =\sum_{m_{1}\ge\cdots\ge m_{r}\ge1}\frac{1}{m_{1}^{k_{1}}\cdots m_{r}^{k_{r}}}\lim_{R\to\infty}\sum_{c_{1}+\cdots+c_{m_{r}}=R}\prod_{s=1}^{m_{r}}\frac{1}{s^{lc_{s}}}\qquad(c_{s}\coloneqq\#\{j:n_{j}=s\})\\
 & =\sum_{m_{1}\ge\cdots\ge m_{r}\ge1}\frac{1}{m_{1}^{k_{1}}\cdots m_{r}^{k_{r}}}\sum_{c_{2},\dots,c_{m_{1}}=0}^{\infty}\prod_{s=2}^{m_{r}}\frac{1}{s^{lc_{s}}}\\
 & =\sum_{m_{1}\ge\cdots\ge m_{r}\ge1}\frac{1}{m_{1}^{k_{1}}\cdots m_{r}^{k_{r}}}\prod_{s=2}^{m_{r}}\sum_{c=0}^{\infty}\frac{1}{s^{lc}}\\
 & =\sum_{m_{1}\ge\cdots\ge m_{r}\ge1}\frac{1}{m_{1}^{k_{1}}\cdots m_{r}^{k_{r}}}\prod_{s=2}^{m_{r}}\frac{s^{l}}{s^{l}-1}.\qedhere
\end{align*}
\end{proof}

\begin{proof}[Proof of Theorem \ref{special_values} (1)]
By Lemma \ref{lem:linf}, we have
\begin{align*}
 \zeta^{\star}(k_{1},\dots,k_{r-1},k_{r}+1,\{1\}^{\infty})
 &=\sum_{m_{1}\ge\cdots\ge m_{r}\ge1}
  \frac{1}{m_{1}^{k_{1}}\cdots m_{r-1}^{k_{r-1}}  m_{r}^{k_{r}+1}}\prod_{s=2}^{m_{r}}\frac{s}{s-1}\\
 & =\sum_{m_{1}\ge\cdots\ge m_{r}\ge1}\frac{1}{m_{1}^{k_{1}}\cdots m_{r}^{k_{r}}}\\
 & =\zeta^{\star}(k_{1},\dots,k_{r}).\qedhere
\end{align*}
\end{proof}

\begin{proof}[Proof of Theorem \ref{special_values} (2)]
Let $L(k_1,\dots,k_r)$ (resp.\ $R(k_1,\dots,k_r)$) be the left (resp. right) hand side of the theorem. By Lemma \ref{lem:linf}, we have
\begin{align*}
L(k_{1},\dots,k_{r}) & =\sum_{m_{1}\geq\cdots\geq m_{r}\ge1}\frac{1}{m_{1}^{k_{1}}\cdots m_{r}^{k_{r}}}\prod_{m=2}^{m_{r}}\frac{m^{2}}{m^{2}-1}\\
 & =2\sum_{m_{1}\geq\cdots\geq m_{r}\ge1}\frac{1}{m_{1}^{k_{1}}\cdots m_{r-1}^{k_{r-1}}}\cdot\frac{1}{m_{r}^{k_{r}-1}(m_{r}+1)}.
\end{align*}
Thus,
\begin{align*}
 & L(k_{1},\dots,k_{r},a)+L(k_{1},\dots,k_{r},a+1)\\
 & =2\sum_{m_{1}\geq\cdots\geq m_{r}\geq n\geq1}\frac{1}{m_{1}^{k_{1}}\cdots m_{r}^{k_{r}}}\cdot\frac{1}{n+1}\left(\frac{1}{n^{a-1}}+\frac{1}{n^{a}}\right)\\
 & =2\sum_{m_{1}\geq\cdots\geq m_{r}\geq n\geq1}\frac{1}{m_{1}^{k_{1}}\cdots m_{r}^{k_{r}}n^{a}}\\
 & =2\zeta^{\star}(k_{1},\dots,k_{r},a).
\end{align*}
On the other hand, by definition,
\[
R(k_{1},\dots,k_{r},a)+R(k_{1},\dots,k_{r},a+1)=2\zeta^{\star}(k_{1},\dots,k_{r},a).
\]
Thus we have
\[
L(k_1,\dots,k_r,a) + L(k_1,\dots,k_r,a+1) 
= R(k_1,\dots,k_r,a) + R(k_1,\dots,k_r,a+1).
\]
Furthermore, we have
\begin{align*}
L(k_1,\dots,k_r,2) &= L(k_1,\dots,k_r),\\
R(k_1,\dots,k_r,2) &= R(k_1,\dots,k_r),    
\end{align*}
and $L(2)=2=R(2)$. 
Thus the claim follows by induction.
\end{proof}

\begin{proof}[Proof of Theorem \ref{special_values} (3)]
We have
\begin{align*}
 \zeta^{\star}(\{k\}^{\infty}) 
 %=\lim_{r\to\infty}\sum_{m_1\ge\cdots\ge m_r\ge1}
 % \frac{1}{m_1^k\cdots m_r^{k}}
 =\sum_{m_{1}\geq m_{2}\geq\cdots\ge1}
  \frac{1}{m_{1}^{k}m_{2}^{k}\cdots}
 =\prod_{m=2}^{\infty}\sum_{s=0}^{\infty}\frac{1}{m^{ks}}
 =\prod_{m=2}^{\infty} \frac{m^{k}}{m^{k}-1}.
\end{align*} 
Since
\[
 m^k-1
 =\prod_{c^k=1} ( m-c ),
\]
we have
\begin{align*}
 \prod_{m=2}^{\infty}\frac{m^{k}}{m^{k}-1}
 &=\lim_{n\to \infty}\prod_{m=2}^{n}\prod_{c^k=1}\frac{m}{m-c}\\
 &=\prod_{c^k=1}\lim_{n\to \infty}\prod_{m=2}^{n}\frac{2+(m-2)}{2-c+(m-2)}.
\end{align*}
Using
\[
 \Gamma(x)
 =\lim_{n\to\infty} \frac{n^x n!}{x(x+1)\cdots(x+n)},
\]
we obtain
\begin{align*}
 \prod_{m=2}^{\infty}\frac{m^{k}}{m^{k}-1}
 =\prod_{c^k=1}\frac{\Gamma(2-c)}{\Gamma(2)}
 =\prod_{c^k=1}\Gamma(2-c).&\qedhere
\end{align*}
\end{proof}

\begin{proof}[Proof of Theorem \ref{special_values} (4)]
It is known that
$\zeta^{\star}(\{2,\{1\}^{n-2}\}^{a},1)=n\zeta(an+1)$ (see \cite{OW} and \cite{Zlo}).
Thus, we have
\[
\zeta^{\star}(\{2,\{1\}^{n-2}\}^{\infty})
=\lim_{a\to\infty}\zeta^{\star}(\{2,\{1\}^{n-2}\}^{a},1)=n.\qedhere
\]
\end{proof}

\begin{proof}[Proof of Theorem \ref{special_values} (5)]
Using \cite[Theorem 1.2]{Zha}, we have
\begin{align*}
 \zeta^{\star}(\{\{2\}^{n},1\}^{d}) 
 &=\sum_{m_{1}\geq\cdots\geq m_{d}\geq1}\frac{2^{\#\{m_{1},\dots,m_{d}\}}}{m_{1}^{2n+1}\cdots m_{d}^{2n+1}} \\
 & =\sum_{\substack{(c_1,c_2,\dots)\in\mathbb{Z}_{\geq0}^{\infty}\\\sum_{m=1}^{\infty}c_{m}=d}}
  \prod_{\substack{m\ge 1 \\ c_{m}\ge 1}}
  \frac{2}{m^{(2n+1)c_{m}}},
\end{align*}
where we put $c_{m}=\#\{j:m_{j}=m\}$.
Then
\begin{align*}
 \zeta^{\star}(\{\{2\}^{n},1\}^{d}) 
 =\sum_{\substack{(c_{2},c_{3},\dots)\in\mathbb{Z}_{\geq0}^{\infty}\\
 \sum_{m=2}^{\infty}c_{m}=d}}
  \prod_{\substack{m\ge 2 \\ c_{m}\ge 1}}
  \frac{2}{m^{(2n+1)c_{m}}}
 +2\sum_{\substack{(c_{2},c_{3},\dots)\in\mathbb{Z}_{\geq0}^{\infty}\\
 \sum_{m=2}^{\infty}c_{m}<d}}
  \prod_{\substack{m\ge 2 \\ c_{m}\ge 1}}
  \frac{2}{m^{(2n+1)c_{m}}}.
\end{align*}

Since
\begin{align*}
 \lim_{d\to\infty}
 \sum_{\substack{(c_{2},c_{3},\dots)\in\mathbb{Z}_{\geq0}^{\infty}\\
  \sum_{m=2}^{\infty}c_{m}=d}}
 \prod_{\substack{m\ge 2 \\ c_{m}\ge 1}}
 \frac{2}{m^{(2n+1)c_{m}}}
 =0
\end{align*}
and
\begin{align*}
 \lim_{d\to\infty}
 \sum_{\substack{(c_{2},c_{3},\dots)\in\mathbb{Z}_{\geq0}^{\infty}\\
 \sum_{m=2}^{\infty}c_{m}<d}}
  \prod_{\substack{m\ge 2 \\ c_{m}\ge 1}}
  \frac{2}{m^{(2n+1)c_{m}}}
 %%%
 &=
 \sum_{(c_{2},c_{3},\dots)\in\mathbb{Z}_{\geq0}^{\infty}}
  \prod_{\substack{m\ge 2 \\ c_{m}\ge 1}}
  \frac{2}{m^{(2n+1)c_{m}}}\\
 %%%
 &=\prod_{m=2}^{\infty}
  \left(
   1+2\sum_{c=1}^{\infty} \frac{1}{m^{(2n+1)c}}
  \right),
\end{align*}
we have
\begin{align*}
 \zeta^{\star}(\{\{2\}^{n},1\}^{\infty})
% &=2\prod_{m=2}^{\infty}
%  \left(1+2\left(\frac{1}{m^{2n+1}}+\frac{1}{m^{2(2n+1)}}+\cdots\right)\right)\\
 &=2\prod_{m=2}^{\infty}
  \frac{m^{2n+1}+1}{m^{2n+1}-1}.
\end{align*}
We obtain the result by a similar calculation as in the proof of Theorem \ref{special_values} (3).
\end{proof}

\begin{proof}[Proof of Theorem \ref{special_values} (6)]
Using the equation \cite[Theorem 4.8 (2-c-2-1)]{Zha} with $c_1=\cdots=c_r=3$ and $a_1=\cdots=a_r=b_1=\cdots=b_r=n$, we have
\[
\zeta^{\star}(\{\{2\}^{n},3,\{2\}^{n},1\}^{d})
=\sum_{m_{1}\geq\cdots\geq m_{2d}\ge1}\frac{(-1)^{m_{1}+\cdots+m_{2d}}2^{\#\{m_{1},\dots,m_{2d}\}}}{m_{1}^{2n+2}\cdots m_{2d}^{2n+2}}.
\]
Using this equality, we have
\begin{align*}
 \lim_{d\to\infty}\zeta^{\star}(\{\{2\}^{n},3,\{2\}^{n},1\}^{d}) 
 & =2\prod_{m=2}^{\infty}
  \left(1+2\left(\frac{(-1)^{m}}{m^{2n+2}}+\frac{(-1)^{2m}}{m^{2(2n+2)}}+\cdots\right)\right)\\
 & =2\prod_{m=2}^{\infty}\left(1+\frac{2}{(-1)^{m}m^{2n+2}-1}\right)\\
 & =2\prod_{m=2}^{\infty}\left(\frac{m^{2n+2}+(-1)^{m}}{m^{2n+2}-(-1)^{m}}\right).
\end{align*}
Then we find
\begin{align*}
 \lim_{d\to\infty}\zeta^{\star}(\{\{2\}^{n},3,\{2\}^{n},1\}^{d}) 
 & =2\prod_{m=2}^{\infty}\left(\frac{m^{2n+2}-1}{m^{2n+2}+1}\right)\times\prod_{m=2:\mathrm{even}}^{\infty}\left(\frac{m^{2n+2}+1}{m^{2n+2}-1}\right)^{2}\\
 & =2\prod_{m=2}^{\infty}\left(\frac{m^{2n+2}-1}{m^{2n+2}+1}\right)\times\prod_{m=1}^{\infty}\left(\frac{m^{2n+2}+(1/2)^{2n+2}}{m^{2n+2}-(1/2)^{2n+2}}\right)^{2}.
\end{align*}
By a similar calculation as in Proof of Theorem \ref{special_values} (3), we get
\begin{align*}
 \lim_{d\to\infty}\zeta^{\star}(\{\{2\}^{n},3,\{2\}^{n},1\}^{d}) 
 & =2\frac{\prod_{c^{2n+2}=1}\Gamma(2-c)^{-1}\Gamma(1-\frac{c}{2})^{2}}{\prod_{c^{2n+2}=-1}\Gamma(2-c)^{-1}\Gamma(1-\frac{c}{2})^{2}}\\
 & =2\prod_{s\in\{\pm1\}}\prod_{c^{2n+2}=s}\Gamma(2-c)^{-s}\Gamma\left(1-\frac{c}{2}\right)^{2s}.\qedhere
\end{align*}
\end{proof}

%%%%%%%%%%%%%%%%%%%%%%%%%%%%%%%%%%%%%%%%%%%%%%%%%%%%%%%%%%%%%%%%%%%%%%%%%%%%%%%%%%%%%%%%%%%%%%%%%%%%%%%
\section{Order property and continuity of the zeta-star map} \label{order_property}
In this section, we will give a proof of Theorem \ref{main2}.
\begin{lem} \label{aaaaa}
 For positive integers $a,b,A$ with $A\ge2$, we have
\begin{align*}
 &\sum_{
  m_{1}\ge\cdots\ge m_{a}\ge
  n_{1}\ge\cdots\ge n_{b}\ge A}
  \frac{1}{m_1^{2}m_{2}\cdots m_{a}
           n_{1}^{k_1}\cdots n_{b}^{k_b}}\leq
  \sum_{
  n_{1}\ge\cdots\ge n_{b}\ge A}
  \frac{1}{(n_1-1)
           n_{1}^{k_1}\cdots n_{b}^{k_b}},\\
 %%%
 &\sum_{
  m_{1}\ge\cdots\ge m_{a}\ge
  n_{1}\ge\cdots\ge n_{b}\ge A}
  \frac{1}{(m_1-1)m_1^{2}m_{2}\cdots m_{a}
           n_{1}^{k_1}\cdots n_{b}^{k_b}}\\
 &\qquad\qquad\qquad\qquad\qquad\qquad\qquad\qquad\qquad\leq
  \left(\frac{A+1}{2A}\right)^a\sum_{
  n_{1}\ge\cdots\ge n_{b}\ge A}
  \frac{1}{(n_1-1)
           n_{1}^{k_1+1}n_2^{k_2}\cdots n_{b}^{k_b}}.
\end{align*}
\end{lem}
\begin{proof}
 We have
\begin{align*}
 &\text{L.H.S. of the first equality}\\
 &\le \sum_{
  m_{1}\ge\cdots\ge m_{a}\ge
  n_{1}\ge\cdots\ge n_{b}\ge A}
  \frac{1}{(m_1-1)m_1m_{2}\cdots m_{a}
           n_{1}^{k_1}\cdots n_{b}^{k_b}}\\
 %%%
 &=\sum_{
  m_{1}\ge\cdots\ge m_{a}\ge
  n_{1}\ge\cdots\ge n_{b}\ge A}
  \left(\frac{1}{m_1-1}-\frac{1}{m_1}\right)\cdot\frac{1}{m_{2}\cdots m_{a}
           n_{1}^{k_1}\cdots n_{b}^{k_b}}\\
 %%%
 &=\sum_{
  m_{2}\ge\cdots\ge m_{a}\ge
  n_{1}\ge\cdots\ge n_{b}\ge A}
  \frac{1}{(m_2-1)m_2 m_{3}\cdots m_{a}
           n_{1}^{k_1}\cdots n_{b}^{k_b}}.
\end{align*}
Repeating the similar calculations, we obtain the first result.
As for the second equality, we have
\begin{align*}
 &\text{L.H.S. of the second equality}\\
 &\le \frac{A+1}{A}\sum_{
  m_{1}\ge\cdots\ge m_{a}\ge
  n_{1}\ge\cdots\ge n_{b}\ge A}
  \frac{1}{(m_1-1)m_1(m_1+1) m_{2}\cdots m_{a}
           n_{1}^{k_1}\cdots n_{b}^{k_b}}\\
 %%%
 &=\frac{A+1}{A}\sum_{
  m_{1}\ge\cdots\ge m_{a}\ge
  n_{1}\ge\cdots\ge n_{b}\ge A}
  \frac{1}{2}\left(\frac{1}{m_1(m_1-1)}-\frac{1}{m_1(m_1+1)}\right)\cdot\frac{1}{m_{2}\cdots m_{a}
           n_{1}^{k_1}\cdots n_{b}^{k_b}}\\
 %%%
 &=\frac{A+1}{2A}\sum_{
  m_{2}\ge\cdots\ge m_{a}\ge
  n_{1}\ge\cdots\ge n_{b}\ge A}
  \frac{1}{(m_2-1)m_2^2 m_{3}\cdots m_{a}
           n_{1}^{k_1}\cdots n_{b}^{k_b}}.
\end{align*}
By repeating the above procedure, we obtain the second result.
\end{proof}
\begin{proof}[Proof that the map $\zeta^{\star}$ is order-reversing]
Let ${\bf k}=(k_{1},k_{2},\dots)\in \mathbb{Z}_{\ge1}^{\infty}$ 
and ${\bf k}'=(k_{1}',k_{2}',\dots)\in \mathbb{Z}_{\ge1}^{\infty}$.
Put ${\bf k}_+=(k_{1}+1,k_{2},k_{3},\dots)$ for ${\bf k}$ and ${\bf k}'_+$ in the same manner.
Assume that ${\bf k}\prec{\bf k'}$ by the lexicographic order. Then
there exists $r\ge1$ such that $k_{i}=k_{i}'$ for $1\le i<r$
and $k_{r}<k_{r}'$. 
Then
\begin{align*}
\zeta^{\star}({\bf k}_+) & =\sum_{m_{1}\geq m_{2}\geq\cdots}\frac{1}{m_{1}^{k_{1}+1}m_{2}^{k_{2}}m_{3}^{k_{3}}\cdots}\\
 & >\sum_{m_{1}\geq m_{2}\geq\cdots\geq m_{r}\geq1}
 \frac{1}{m_{1}^{k_{1}+1} m_{2}^{k_{2}}\cdots m_{r}^{k_{r}}}\\
 & =\zeta^{\star}(k_{1}+1,k_{2},\dots,k_{r})
\end{align*}
and
\begin{align*}
\zeta^{\star}({\bf k}'_+)
&\leq\zeta^{\star}(k_{1}+1,k_{2},\dots, k_{r-1},k_{r}',\{1\}^{\infty})\\
&=\zeta^{\star}(k_{1}+1,k_{2},\dots, k_{r-1},k_{r}'-1)\leq\zeta^{\star}(k_{1}+1,k_{2},\dots,k_{r}). 
\end{align*}
Thus, we have $\zeta^{\star}({\bf k}_+)>\zeta^{\star}({\bf k}'_+)$, i.e., $\zeta^{\star}$ is an order-reversing map.
\end{proof}

\begin{proof}[Proof that $\zeta^{\star}({\bf k})$ is convergent for ${\bf k}\neq (2,\{1\}^{\infty})$]
Let ${\bf k}=(k_1,k_2,\dots)\in\mathbb{Z}_{\ge1}^{\infty}$ with $k_1 \geq 2$ and ${\bf k}\neq (2,\{1\}^{\infty})$.
Then there exists $n\geq 2$ such that $(\{2,\{1\}^{n-2}\}^{\infty}) \prec {\bf k}$. Thus, by Theorem \ref{special_values} (4), $\zeta({\bf k}) \leq \zeta(\{2,\{1\}^{n-2}\}^{\infty})=n$, which implies the convergence of $\zeta({\bf k})$.
\end{proof}

\begin{proof}[Proof that the map $Z^{\star}$ is continuous]
Let ${\bf k}=(k_{1},k_{2},\dots)\in \mathbb{Z}_{\ge1}^{\infty}$ with $k_{1}\ge2$.
We need to show that for any $\epsilon>0$,
there exists ${\bf l}$ and ${\bf l}'$ such that
\begin{align*}
{\bf k}\prec{\bf k}'\prec{\bf l}
&\implies
\zeta^{\star}({\bf k})-\zeta^{\star}({\bf k}')<\epsilon,\\
{\bf l}'\prec{\bf k}'\prec{\bf k}
&\implies
\zeta^{\star}({\bf k}')-\zeta^{\star}({\bf k})<\epsilon.
\end{align*}

Since the sequence $\left(\zeta^{\star}(k_1,\dots,k_n)\right)_{n=1}^\infty$ is bounded and monotone increasing, there exist $r\geq1$ such that
\[
\zeta^{\star}(k_{1},\dots,k_{r})>\zeta^{\star}({\bf k})-\epsilon.
\]
Thus
\[
\zeta^{\star}(k_{1},\dots ,k_{r-1},k_{r}+1,\{1\}^{\infty})>\zeta^{\star}({\bf k})-\epsilon.
\]
By taking ${\bf l}=(k_{1},\dots ,k_{r-1},k_{r}+1,\{1\}^{\infty})$, we obtain the first line.

Now we will show the second line.
We first show the claim for indices with a finite number of elements greater than or equal to $2$.
Let ${\bf k}=(k_{1},\dots,k_{r},\{1\}^\infty)$ with $k_{r}\ge2$.
Since
\begin{align*}
    &\lim_{a\to\infty}
     \zeta^\star(k_1,\dots,k_{r-1},k_r-1,a+1,\{1\}^\infty)\\
    &=\lim_{a\to\infty}\sum_{m_{1}\ge\cdots\ge m_{r}\ge1}\frac{1}{m_{1}^{k_{1}}\cdots m_{r-1}^{k_{r-1}}m_{r}^{k_{r}-1}}\sum_{s=1}^{m_{r}}\frac{1}{s^{a}}\\
    &\leq \lim_{a\to\infty}\zeta(a) \sum_{m_{1}\ge\cdots\ge m_{r}\ge1}\frac{1}{m_{1}^{k_{1}}\cdots m_{r-1}^{k_{r-1}}m_{r}^{k_{r}-1}}\\
    &=\zeta^\star(k_{1},\dots,k_{r-1},k_{r}-1)\\
    &=\zeta^\star({\bf k}),
\end{align*}
there exist $n\geq1$ such that
\[
\zeta^{\star}(k_1,\dots,k_{r-1},k_{r}-1,n+1,\{1\}^{\infty})
<\zeta^{\star}({\bf k})+\epsilon.
\]
Thus, the claim holds for indices with a finite number of elements greater than or equal to $2$.

Assume that there exists infinitely many $j$ such that $k_{j}>1$.
\[
{\bf k}=(a_{1},\{1\}^{b_{1}},a_{2},\{1\}^{b_{2}},\dots)\qquad(a_{j}\geq2,b_{j}\geq0).
\]
We need to show the existence of $s$ such that
\[
\zeta^{\star}(a_{1},\{1\}^{b_{1}},a_{2},\{1\}^{b_{2}},\dots ,a_{s-1},\{1\}^{b_{s-1}},a_{s}-1)<\zeta^{\star}({\bf k})+\epsilon.
\]
Note that 
\begin{align*}
 & \zeta^{\star}(a_{1},\{1\}^{b_{1}},a_{2},\{1\}^{b_{2}},\dots ,a_{s-1},\{1\}^{b_{s-1}},a_{s}-1)-\zeta^{\star}({\bf k})\\
 & <\zeta^{\star}(a_{1},\{1\}^{b_{1}},a_{2},\{1\}^{b_{2}},\dots ,a_{s-1},\{1\}^{b_{s-1}},a_{s}-1)\\
  &\quad
  -\zeta^{\star}(a_{1},\{1\}^{b_{1}},a_{2},\{1\}^{b_{2}},\dots ,a_{s-1},\{1\}^{b_{s-1}},a_{s})\\
 & \leq\zeta^{\star}(2,\{1\}^{b},2,\{1\}^{r-2},1)-\zeta^{\star}(2,\{1\}^{b},2,\{1\}^{r-2},2),
\end{align*}
where $b=b_1$ and $r=b_2+\cdots+b_{s-1}+s-1$.
Then we have
\begin{align*}
 &\zeta^{\star}(a_{1},\{1\}^{b_{1}},a_{2},\{1\}^{b_{2}},\dots ,a_{s-1},\{1\}^{b_{s-1}},a_{s}-1)
  -\zeta^{\star}({\bf k})\\
 &<\sum_{m\geq n_{1}\geq\cdots\geq n_{b+r}\ge1}
  \frac{1}{m^{2}n_{1}\cdots n_{b}n_{b+1}^{2}n_{b+2}\cdots n_{b+r-1}}\left(\frac{1}{n_{b+r}}
  -\frac{1}{n_{b+r}^{2}}\right)\\
 & =\sum_{m\geq n_{1}\geq\cdots\geq n_{b+r}\geq2}\frac{1}{m^{2}n_{1}\cdots n_{b}n_{b+1}^{2}n_{b+2}\cdots n_{b+r-1}}\left(\frac{1}{n_{b+r}}-\frac{1}{n_{b+r}^{2}}\right)\\
 & \leq\sum_{m\geq n_{1}\geq\cdots\geq n_{b+r}\geq2}\frac{1}{m^{2}n_{1}\cdots n_{b}n_{b+1}^{2}n_{b+2}\cdots n_{b+r}}\\
 &\le\left(\frac{3}{4}\right)^{r-1}
  \sum_{n_{b+r}\ge 2}
  \frac{1}{(n_{b+r}-1)n_{b+r}^{2}}.
\end{align*}
Here, we used Lemma \ref{aaaaa} for the last inequality. 
Thus, for any $\epsilon>0$, there exists $s$ such that
\[
\zeta^{\star}(a_{1},\{1\}^{b_{1}},a_{2},\{1\}^{b_{2}},\dots a_{s-1},\{1\}^{b_{s-1}},a_{s}-1)<\zeta^{\star}({\bf k})+\epsilon.
\]
This finishes the proof.
\end{proof}

%%%%%%%%%%%%%%%%%%%%%%%%%%%%%%%%%%%%%%%%%%%%%%%%%%%%%%%%%%%%%%%%%%%%%%%%%%%%%%%%%%%%%%%%%%%%%%%%%%%%%%%
\section{Analytic properties of the zeta-star map}
This section investigates the differential of $Z^{\star}$.
Hereinafter, we understand $0^0=1$.
\begin{lem}
For $z=\sum_{j=1}^{\infty}\frac{a_{j}}{2^{j}}$ with $a_{j}\in\{0,1\}$,
we have 
\[
Z^{\star}(z)=\sum_{m_{1}\geq m_{2}\geq\cdots\geq1}\frac{a_{1}^{m_{1}-m_{2}}a_{2}^{m_{2}-m_{3}}\cdots}{m_{1}^{2}m_{2}m_{3}\cdots}.
\]
\end{lem}
\begin{proof}
Note that $a_{1}^{m_{1}-m_{2}}a_{2}^{m_{2}-m_{3}}\cdots$ vanishes except
for the case $m_{j}=m_{j+1}$ for all $j$ such that $a_{j}=0$. 
The case where there exists infinitely many $j$ such that $a_{j}=1$ follows from the definition of $Z^{\star}$. 
The case $z=0$ also follows from the definition of $Z^{\star}$. 
The other case follows from Theorem \ref{special_values} (1), e.g., when $a_{j}=\delta_{j,1}$,
\begin{align*}
 Z^{\star}\left(\frac{1}{2}+\frac{0}{4}+\frac{0}{8}+\cdots\right) 
 &=Z^{\star}\left(\frac{0}{2}+\frac{1}{4}+\frac{1}{8}+\cdots\right)
 =\zeta^{\star}(3,\{1\}^{\infty})\\
 &=\zeta^{\star}(2)
 =\sum_{m_{1}\geq m_{2}\geq\cdots\geq1}\frac{1^{m_{1}-m_{2}}0^{m_{2}-m_{3}}0^{m_{3}-m_{4}}}{m_{1}^{2}m_{2}m_{3}m_{4}\cdots}.\qedhere
\end{align*}
\end{proof}

\begin{lem} \label{42}
For $z=\sum_{j=1}^{\infty}\frac{a_{j}}{2^{j}}$ with $a_{j}\in\{0,1\}$, we have
\[
Z^{\star}(z)=1+\frac{z}{2}+\sum_{d=1}^{\infty}a_{d}\left(\sum_{m_{1}\geq\cdots\geq m_{d}\geq3}\frac{a_{1}^{m_{1}-m_{2}}\cdots a_{d-1}^{m_{d-1}-m_{d}}}{m_{1}^{2}m_{2}\cdots m_{d}}\right)2^{d}\left(z-\sum_{i=1}^{d-1}\frac{a_{i}}{2^{i}}\right).
\]
\end{lem}
\begin{proof}
It follows from the following calculation
\begin{align*}
Z^{\star}(z) & =\sum_{m_{1}\geq m_{2}\geq\cdots\geq1}\frac{a_{1}^{m_{1}-m_{2}}a_{2}^{m_{2}-m_{3}}\cdots}{m_{1}^{2}m_{2}m_{3}\cdots}\\
 & =\sum_{1\leq d\leq e}\sum_{\substack{m_{1}\geq\cdots\geq m_{d}\geq3\\
2=m_{d+1}=\cdots=m_{e}\\
1=m_{e+1}=m_{e+2}=\cdots
}
}\frac{a_{1}^{m_{1}-m_{2}}a_{2}^{m_{2}-m_{3}}\cdots}{m_{1}^{2}m_{2}m_{3}\cdots}+\sum_{2\geq m_{1}\geq m_{2}\geq\cdots}\frac{a_{1}^{m_{1}-m_{2}}a_{2}^{m_{2}-m_{3}}\cdots}{m_{1}^{2}m_{2}m_{3}\cdots}.
\end{align*}
Here,
\begin{align*}
 & \sum_{1\leq d\leq e}\sum_{\substack{m_{1}\geq\cdots\geq m_{d}\geq3\\
2=m_{d+1}=\cdots=m_{e}\\
1=m_{e+1}=m_{e+2}=\cdots
}
}\frac{a_{1}^{m_{1}-m_{2}}a_{2}^{m_{2}-m_{3}}\cdots}{m_{1}^{2}m_{2}m_{3}\cdots}\\
 & =\sum_{d=1}^{\infty}\left(\sum_{m_{1}\geq\cdots\geq m_{d}\geq3}\frac{a_{1}^{m_{1}-m_{2}}\cdots a_{d-1}^{m_{d-1}-m_{d}}}{m_{1}^{2}m_{2}\cdots m_{d}}\right)\sum_{e=d}^{\infty}\frac{a_{d}a_{e}}{2^{e-d}}\\
 & =\sum_{d=1}^{\infty}a_{d}\left(\sum_{m_{1}\geq\cdots\geq m_{d}\geq3}\frac{a_{1}^{m_{1}-m_{2}}\cdots a_{d-1}^{m_{d-1}-m_{d}}}{m_{1}^{2}m_{2}\cdots m_{d}}\right)2^{d}\left(z-\sum_{i=1}^{d-1}\frac{a_{i}}{2^{i}}\right)
\end{align*}
and
\begin{align*}
 &\sum_{2\geq m_{1}\geq m_{2}\geq\cdots}\frac{a_{1}^{m_{1}-m_{2}}a_{2}^{m_{2}-m_{3}}\cdots}{m_{1}^{2}m_{2}m_{3}\cdots}\\
 &=1+\frac{1}{2}\sum_{\substack{2\geq m_{1}\geq m_{2}\geq\cdots\\m_{1}\neq1}}
 \frac{a_{1}^{m_{1}-m_{2}}a_{2}^{m_{2}-m_{3}}\cdots}{m_{1}m_{2}m_{3}\cdots}\\
 &=1+\frac{1}{2}\sum_{e=1}^{\infty}\sum_{\substack{2=m_{1}=\cdots=m_{e}\\1=m_{e+1}=m_{e+2}=\cdots}}
 \frac{a_{1}^{m_{1}-m_{2}}a_{2}^{m_{2}-m_{3}}\cdots}{m_{1}m_{2}m_{3}\cdots}\\
 &=1+\frac{1}{2}\sum_{e=1}^{\infty}\frac{a_{e}}{2^{e}}\\
 &=1+\frac{z}{2}. \qedhere
\end{align*}
\end{proof}

\begin{lem} \label{cc41}
For $s\ge1$, we have
\[
 \sum_{m_{1}\geq\cdots\geq m_{s}\geq3}
 \frac{1}{m_{1}^{4}m_{2}\cdots m_{s}}
 =O\left(\frac{s}{3^s}\right).
\]
\end{lem}
\begin{proof}
For $x>0$, put
\[
F_{s}(x)=\sum_{m_{1}\geq\cdots\geq m_{s}\geq3}\frac{1}{m_{2}\cdots m_{s}}
\begin{cases}
\frac{1}{m_{1}(m_{1}-1)(m_{1}-2)(m_{1}-3)} & \text{if }m_{1}>3,\\
x & \text{if }m_{1}=3.
\end{cases}
\]
Note that
\begin{align*}
\sum_{m=n}^{\infty}\frac{1}{m(m-1)(m-2)(m-3)} 
 & =\frac{1}{3}\sum_{m=n}^{\infty}\left(\frac{1}{(m-1)(m-2)(m-3)}-\frac{1}{m(m-1)(m-2)}\right)\\
 & =\frac{1}{3(n-1)(n-2)(n-3)}.
\end{align*}
Then
\begin{align*}
 F_{s}(x) 
 &=\sum_{\substack{m_{2}\geq\cdots\geq m_{s}\geq3\\m_{2}>3}}
  \frac{1}{m_{2}\cdots m_{s}}\sum_{m_{1}=m_{2}}^{\infty}\frac{1}{m_{1}(m_{1}-1)(m_{1}-2)(m_{1}-3)}\\
 &\quad +\quad\sum_{\substack{m_{2}\geq\cdots\geq m_{s}\geq3\\m_{2}=3}}
  \frac{1}{m_{2}\cdots m_{s}}\left(\sum_{m_{1}=4}^{\infty}\frac{1}{m_{1}(m_{1}-1)(m_{1}-2)(m_{1}-3)}+x\right)\\
 &=\frac{1}{3}\sum_{\substack{m_{2}\geq\cdots\geq m_{s}\geq3\\m_{2}>3}}
  \frac{1}{m_{3}\cdots m_{s}}\cdot\frac{1}{m_{2}(m_{2}-1)(m_{2}-2)(m_{2}-3)}\\
 &\quad +\frac{1}{3}\sum_{\substack{m_{2}\geq\cdots\geq m_{s}\geq3\\m_{2}=3}}
  \frac{1}{m_{3}\cdots m_{s}}\left(\frac{1}{18}+x\right)\\
 &=\frac{1}{3}F_{s-1}\left(x+\frac{1}{18}\right).
\end{align*}
Thus we have
\[
F_{s}(x)=\frac{1}{3^{s-1}}F_{1}\left(x+\frac{s-1}{18}\right).
\]
Hence we get
\begin{align*}
\sum_{m_{1}\geq\cdots\geq m_{s}\geq3}\frac{1}{m_{1}^{4}m_{2}\cdots m_{s}}
\leq F_{s}\left(\frac{1}{18}\right).
%&=\frac{1}{3^{s-1}}F_{1}\left(\frac{s}{18}\right)
%=\frac{s}{18}+\sum_{m>3}\frac{1}{m(m-1)(m-2)(m-3)}.
%=\frac{1}{3^{s-1}}\cdot\frac{s+1}{18}.
\end{align*}
This finishes the proof. 
\end{proof}
\begin{lem} \label{43}
Let $a_{j}\in\{0,1\}$ for $1\le j\le d-1$.
Assume that $\sum_{i=1}^{t}(1-a_{i})\geq2$ for some $t\le d-1$, then
\[
\sum_{m_{1}\geq\cdots\geq m_{d}\geq3}\frac{a_{1}^{m_{1}-m_{2}}\cdots a_{d-1}^{m_{d-1}-m_{d}}}{m_{1}^{2}m_{2}\cdots m_{d}}=O\left(\frac{d-t}{3^{d-t}}\right).
\]
\end{lem}
\begin{proof}
From the assumption, let $a_u=a_v=0 \; (u<v\le t)$ 
and $a_1=\cdots=a_{u-1}=a_{u+1}=\cdots=a_{v-1}$=1.
Then
\begin{align*}
 \text{L.H.S.}
 &=\sum_{m_{1}\geq\cdots\geq m_{d}\geq3}
  \frac{
  a_{u}^{m_{u}-m_{u+1}} a_{v}^{m_{v}-m_{v+1}}
  a_{v+1}^{m_{v+1}-m_{v+2}}\cdots a_{d-1}^{m_{d-1}-m_{d}}}{m_{1}^{2}m_{2}\cdots m_{d}}\\
 &=\sum_{m_{1}\geq\cdots\geq m_{u}=m_{u+1}\geq \cdots\geq m_{v}=m_{v+1} \geq\cdots\ge m_{d}\geq3}
  \frac{
  a_{v+1}^{m_{v+1}-m_{v+2}}\cdots a_{d-1}^{m_{d-1}-m_{d}}}{m_{1}^{2}m_{2}\cdots m_{d}}.
\end{align*}
Using Lemma \ref{aaaaa}, we have
\begin{align*}
 \text{L.H.S.}
 &\le\frac{3}{2}\sum_{m_{v+1}\geq\cdots\geq m_{d}\geq3}
  \frac{
  a_{v+1}^{m_{v+1}-m_{v+2}}\cdots a_{d-1}^{m_{d-1}-m_{d}}}{m_{v+1}^{4}m_{v+2}\cdots m_{d}}.
\end{align*}
By Lemma \ref{cc41}, we obtain the result.
\end{proof}
\begin{lem} \label{44}
Let $z=\sum_{j=1}^{\infty}\frac{a_{j}}{2^{j}}$ with $a_{j}\in\{0,1\}$ and assume that $\sum_{i=1}^{t}(1-a_{i})\geq2$.
Then we have
\[
 Z^{\star}(z)
 =1+\frac{z}{2}
  +\sum_{d=1}^{r}a_{d}
   \left(\sum_{m_{1}\geq\cdots\geq m_{d}\geq3}
   \frac{a_{1}^{m_{1}-m_{2}}\cdots a_{d-1}^{m_{d-1}-m_{d}}}{m_{1}^{2}m_{2}\cdots m_{d}}\right)2^{d}
   \left(
    z-\sum_{i=1}^{d-1}
    \frac{a_{i}}{2^{i}}
   \right)
   +O\left(\frac{r-t}{3^{r-t}}\right).
\]
\end{lem}
\begin{proof}
Since
\[
 z-\sum_{i=1}^{d-1}\frac{a_{i}}{2^{i}}
 \le\sum_{i=d}^{\infty}\frac{1}{2^{i}}=2^{1-d},
\]
we have
\begin{align*}
 &\sum_{d=r+1}^{\infty}a_{d}
  \left(\sum_{m_{1}\geq\cdots\geq m_{d}\geq3}
  \frac{a_{1}^{m_{1}-m_{2}}\cdots a_{d-1}^{m_{d-1}-m_{d}}}{m_{1}^{2}m_{2}\cdots m_{d}}\right)
  2^{d}\left(z-\sum_{i=1}^{d-1}\frac{a_{i}}{2^{i}}\right)\\
 &\le2\sum_{d=r+1}^{\infty}
  \left(\sum_{m_{1}\geq\cdots\geq m_{d}\geq3}
   \frac{a_{1}^{m_{1}-m_{2}}\cdots a_{d-1}^{m_{d-1}-m_{d}}}{m_{1}^{2}m_{2}\cdots m_{d}}\right)\\
 &\le 2\sum_{d=r+1}^{\infty}O\left(\frac{d-t}{3^{d-t}}\right)
\end{align*}
by Lemma \ref{43}. 
Note that in the above equality, the $O$-constants are independent of $d$.
By Lemma \ref{42}, we get the result.
\end{proof}

\begin{lem} \label{46}
Let $x=\sum_{j=1}^{\infty} a_{j}/2^{j}$ with $a_{j}\in\{0,1\}$ and $y=\sum_{j=1}^{\infty} b_{j}/2^{j}$ with $b_{j}\in\{0,1\}$.
Assume that $a_{j}=b_{j}$ for $j=1,\dots,r$. Furthermore, assume
that $\sum_{i=1}^{t}(1-a_{i})\geq2$ with $t\leq r$. Then
\[
Z^{\star}(x)-Z^{\star}(y)=(x-y)\left(\frac{1}{2}+\sum_{d=1}^{r}a_{d}\left(\sum_{m_{1}\geq\cdots\geq m_{d}\geq3}\frac{a_{1}^{m_{1}-m_{2}}\cdots a_{d-1}^{m_{d-1}-m_{d}}}{m_{1}^{2}m_{2}\cdots m_{d}}\right)2^{d}\right)+O\left(\frac{r-t}{3^{r-t}}\right).
\]
\end{lem}
\begin{proof}
 By Lemma \ref{44}, the proof is clear.
\end{proof}

\begin{thm} \label{47a}
Let $z=\sum_{j=1}^{\infty}\frac{a_{j}}{2^{j}}$ with $a_{j}\in\{0,1\}$. Assume that
$\sum_{j=1}^{\infty}a_{j}=\infty$ and $\sum_{j=1}^{t}(1-a_{j})\geq2$ for some $t$.
Then
\[
\partial_{-}Z^{\star}(z)
:=\lim_{x\to z-0}\frac{Z^{\star}(z)-Z^{\star}(x)}{z-x}=\frac{1}{2}+\sum_{d=1}^{\infty}a_{d}\left(\sum_{m_{1}\geq\cdots\geq m_{d}\geq3}\frac{a_{1}^{m_{1}-m_{2}}\cdots a_{d-1}^{m_{d-1}-m_{d}}}{m_{1}^{2}m_{2}\cdots m_{d}}\right)2^{d}.
\]
Thus, $Z^{\star}$ is left-differentiable at $z$ if $z\not \in\{1-\frac{1}{2^n}\mid n>0\}$.
\end{thm}
\begin{proof}
Take $0<x<z$ and put $x=\sum_{j=1}^{\infty}\frac{b_{j}}{2^{j}}$.
Let $p=p(x)$ be the minimal integer such that $(a_{p},b_{p})=(1,0)$.
Put 
\[
y=\sum_{j=1}^{p}\frac{a_{j}}{2^{j}}=\sum_{j=1}^{p-1}\frac{a_{j}}{2^{j}}+\frac{1}{2^{p+1}}+\frac{1}{2^{p+2}}+\cdots.
\]
Furthermore, let $r$ be the maximal integer such that
\[
a_{p+1}=a_{p+2}=\cdots=a_{r}=0
\]
and
\[
b_{p+1}=b_{p+2}=\cdots=b_{r}=1.
\]
Then by Lemma \ref{46}, we have
\[
Z^{\star}(z)-Z^{\star}(y)=(z-y)\left(\frac{1}{2}+\sum_{d=1}^{r}a_{d}\left(\sum_{m_{1}\geq\cdots\geq m_{d}\geq3}\frac{a_{1}^{m_{1}-m_{2}}\cdots a_{d-1}^{m_{d-1}-m_{d}}}{m_{1}^{2}m_{2}\cdots m_{d}}\right)2^{d}\right)+O\left(\frac{r-t}{3^{r-t}}\right)
\]
and
\[
Z^{\star}(y)-Z^{\star}(x)=(y-x)\left(\frac{1}{2}+\sum_{d=1}^{r}b_{d}\left(\sum_{m_{1}\geq\cdots\geq m_{d}\geq3}\frac{b_{1}^{m_{1}-m_{2}}\cdots b_{d-1}^{m_{d-1}-m_{d}}}{m_{1}^{2}m_{2}\cdots m_{d}}\right)2^{d}\right)+O\left(\frac{r-t}{3^{r-t}}\right).
\]

From the intermediate value theorem, there exists a real number $u(x)$ between 
\[
\frac{1}{2}+\sum_{d=1}^{r}a_{d}\left(\sum_{m_{1}\geq\cdots\geq m_{d}\geq3}\frac{a_{1}^{m_{1}-m_{2}}\cdots a_{d-1}^{m_{d-1}-m_{d}}}{m_{1}^{2}m_{2}\cdots m_{d}}\right)2^{d}
\]
and
\[
\frac{1}{2}+\sum_{d=1}^{r}b_{d}\left(\sum_{m_{1}\geq\cdots\geq m_{d}\geq3}\frac{b_{1}^{m_{1}-m_{2}}\cdots b_{d-1}^{m_{d-1}-m_{d}}}{m_{1}^{2}m_{2}\cdots m_{d}}\right)2^{d}
\]
such that
\[
Z^{\star}(z)-Z^{\star}(x)
=(z-x)u(x)
 +O\left(\frac{r-t}{3^{r-t}}\right).
\]

Since
\begin{align*}
z-x
&\ge\frac{1}{2^p}-\frac{1}{2^{p+1}}-\cdots-\frac{1}{2^r}-\sum_{j=r+2}^{\infty}\frac{1}{2^{j}}\\
&=\frac{1}{2^{r+1}},
\end{align*}
we have
\[
 \frac{Z^{\star}(z)-Z^{\star}(x)}{z-x}
 =u(x)
  +O\left(r(2/3)^r3^t\right).
\]

By the condition $\sum_{j=1}^{\infty}a_{j}=\infty$, we have $\lim_{x\to z-0}p(x)=\infty$
and thus 
\[
\lim_{x\to z-0}u(x)
=\frac{1}{2}
 +\sum_{d=1}^{\infty}a_{d}
  \left(\sum_{m_{1}\geq\cdots\geq m_{d}\geq3}\frac{a_{1}^{m_{1}-m_{2}}\cdots a_{d-1}^{m_{d-1}-m_{d}}}{m_{1}^{2}m_{2}\cdots m_{d}}\right)2^{d},
\]
which completes the proof.
\end{proof}

\begin{thm}\label{48a}
Let $z=\sum_{j=1}^{\infty}\frac{a_{j}}{2^{j}}$ with $a_{j}\in\{0,1\}$ and assume that $\sum_{j=1}^{\infty}(1-a_{j})=\infty$.
Then
\[
\partial_{+}Z^{\star}(z):=\lim_{x\to z+0}\frac{Z^{\star}(z)-Z^{\star}(x)}{z-x}=\frac{1}{2}+\sum_{d=1}^{\infty}a_{d}\left(\sum_{m_{1}\geq\cdots\geq m_{d}\geq3}\frac{a_{1}^{m_{1}-m_{2}}\cdots a_{d-1}^{m_{d-1}-m_{d}}}{m_{1}^{2}m_{2}\cdots m_{d}}\right)2^{d}.
\]
Thus, $Z^{\star}$ is right-differentiable at $z$ for any $0\leq z<1$.
\end{thm}
\begin{proof}
Take $x\in(z,1)$ and put $x=\sum_{j=1}^{\infty}\frac{b_{j}}{2^{j}}$.
Let $p=p(x)$ be the minimal integer such that $(a_{p},b_{p})=(0,1)$.
Put 
\[
 y=\sum_{j=1}^{p}
 \frac{b_{j}}{2^{j}}
 =\sum_{j=1}^{p-1}\frac{b_{j}}{2^{j}}+\frac{1}{2^{p+1}}+\frac{1}{2^{p+2}}+\cdots.
\]
Furthermore, let $r$ be the maximal integer such that
\[
a_{p+1}=a_{p+2}=\cdots=a_{r}=1
\]
and
\[
b_{p+1}=b_{p+2}=\cdots=b_{r}=0.
\]
Then by Lemma \ref{46}, we have
\[
 Z^{\star}(x)-Z^{\star}(y)
 =(x-y)\left(\frac{1}{2}
 +\sum_{d=1}^{r}b_{d}\left(\sum_{m_{1}\geq\cdots\geq m_{d}\geq3}\frac{b_{1}^{m_{1}-m_{2}}\cdots b_{d-1}^{m_{d-1}-m_{d}}}{m_{1}^{2}m_{2}\cdots m_{d}}\right)2^{d}\right)+O\left(\frac{r-t}{3^{r-t}}\right)
\]
and
\[
 Z^{\star}(y)-Z^{\star}(z)
 =(y-z)\left(\frac{1}{2}+\sum_{d=1}^{r}a_{d}\left(\sum_{m_{1}\geq\cdots\geq m_{d}\geq3}\frac{a_{1}^{m_{1}-m_{2}}\cdots a_{d-1}^{m_{d-1}-m_{d}}}{m_{1}^{2}m_{2}\cdots m_{d}}\right)2^{d}\right)+O\left(\frac{r-t}{3^{r-t}}\right).
\]

From the intermediate value theorem, there exists a real number $u(x)$ between 
\[
\frac{1}{2}+\sum_{d=1}^{r}a_{d}\left(\sum_{m_{1}\geq\cdots\geq m_{d}\geq3}\frac{a_{1}^{m_{1}-m_{2}}\cdots a_{d-1}^{m_{d-1}-m_{d}}}{m_{1}^{2}m_{2}\cdots m_{d}}\right)2^{d}
\]
and
\[
\frac{1}{2}+\sum_{d=1}^{r}b_{d}\left(\sum_{m_{1}\geq\cdots\geq m_{d}\geq3}\frac{b_{1}^{m_{1}-m_{2}}\cdots b_{d-1}^{m_{d-1}-m_{d}}}{m_{1}^{2}m_{2}\cdots m_{d}}\right)2^{d}
\]
such that
\[
Z^{\star}(z)-Z^{\star}(x)
=(z-x)u(x)
 +O\left(\frac{r-t}{3^{r-t}}\right).
\]

Since
\begin{align*}
|z-x|
%&=\sum_{j=1}^{\infty}\frac{a_{j}}{2^{j}}-\sum_{j=1}^{\infty}\frac{b_{j}}{2^{j}}\\
%&=\frac{1}{2^p}-\frac{1}{2^{p+1}}-\cdots-\frac{1}{2^r}+\sum_{j=r+1}^{\infty}\frac{a_{j}-b_j}{2^{j}}\\
&\ge\frac{1}{2^p}-\frac{1}{2^{p+1}}-\cdots-\frac{1}{2^r}-\sum_{j=r+2}^{\infty}\frac{1}{2^{j}}\\
%&=\frac{1}{2^p}-\frac{1}{2^{p+1}}-\cdots-\frac{1}{2^{r+1}}\\
&=\frac{1}{2^{r+1}},
\end{align*}
we have
\[
 \frac{Z^{\star}(z)-Z^{\star}(x)}{z-x}
 =u(x)
  +O\left(r(2/3)^r3^t\right).
\]

By the condition $\sum_{j=1}^{\infty}(1-a_{j})=\infty$, we have $\lim_{x\to z+0}p(x)=\infty$
and thus 
\[
\lim_{x\to z+0}u(x)
=\frac{1}{2}
 +\sum_{d=1}^{\infty}a_{d}
  \left(\sum_{m_{1}\geq\cdots\geq m_{d}\geq3}\frac{a_{1}^{m_{1}-m_{2}}\cdots a_{d-1}^{m_{d-1}-m_{d}}}{m_{1}^{2}m_{2}\cdots m_{d}}\right)2^{d},
\]
which completes the proof.
\end{proof}
\begin{rem}\label{rem:left=right}
By Theorems \ref{47a} and \ref{48a}, if $z$ admits a $2$-adic expansion $\sum_{j=1}^{\infty}\frac{a_j}{2^j}$ satisfying $\sum_{j=1}^{\infty}a_j = \sum_{j=1}^{\infty}(1-a_j)=\infty$, then $\partial_{+}Z^{\star}(z)=\partial_{-}Z^{\star}(z)$. This implies Theorem \ref{main3} (4).    
\end{rem}

\begin{thm} \label{cor:ccc}
For $z=\sum_{j=1}^{r}\frac{a_{j}}{2^{j}}$ with $a_{j}\in\{0,1\}$ and $a_{r}=1$, we have
\[
\partial_{+}Z^{\star}(z)=\frac{1}{2}+\sum_{d=1}^{r}a_{d}\left(\sum_{m_{1}\geq\cdots\geq m_{d}\geq3}\frac{a_{1}^{m_{1}-m_{2}}\cdots a_{d-1}^{m_{d-1}-m_{d}}}{m_{1}^{2}m_{2}\cdots m_{d}}\right)2^{d}
\]
and
\[
\partial_{-}Z^{\star}(z)=\partial_{+}Z^{\star}(z)+\sum_{m_{1}\geq\cdots\geq m_{r}\geq3}\frac{a_{1}^{m_{1}-m_{2}}\cdots a_{r-1}^{m_{r-1}-m_{r}}}{m_{1}^{2}m_{2}\cdots m_{r}}2^{r}(m_{r}-2)\qquad(z\neq1-1/2^{r}).
\]
\end{thm}
\begin{proof}
The first statement is just a special case of Theorem \ref{48a}.
Since
\[
z=\sum_{j=1}^{r-1}\frac{a_{j}}{2^{j}}+\frac{1}{2^{r+1}}+\frac{1}{2^{r+2}}+\cdots,
\]
we have
\begin{align*}
\partial_{-}Z^{\star}(z) & =\frac{1}{2}+\sum_{d=1}^{r-1}a_{d}\left(\sum_{m_{1}\geq\cdots\geq m_{d}\geq3}\frac{a_{1}^{m_{1}-m_{2}}\cdots a_{d-1}^{m_{d-1}-m_{d}}}{m_{1}^{2}m_{2}\cdots m_{d}}\right)2^{d}\\
 & \quad+\sum_{d=r+1}^{\infty}\left(\sum_{m_{1}\geq\cdots\geq m_{r}\geq3}\frac{a_{1}^{m_{1}-m_{2}}\cdots a_{r-1}^{m_{r-1}-m_{r}}}{m_{1}^{2}m_{2}\cdots m_{r}}\sum_{m_{r}=m_{r+1}\geq\cdots\geq m_{d}\geq3}\frac{1}{m_{r+1}\cdots m_{d}}\right)2^{d}
\end{align*}
by Lemma \ref{47a}.
Since
\begin{align*}
\frac{2^{r+1}}{m_{r}}\prod_{n=3}^{m_{r}}\sum_{c=0}^{\infty}\left(\frac{2}{n}\right)^{c}
=\frac{2^{r+1}}{m_{r}}\prod_{n=3}^{m_{r}}\frac{n}{n-2}
=2^{r}(m_{r}-1),
\end{align*}
the third term equals
\begin{align*}
 &\sum_{p=0}^{\infty}\left(\sum_{m_{1}\geq\cdots\geq m_{r}\geq3}\frac{a_{1}^{m_{1}-m_{2}}\cdots a_{r-1}^{m_{r-1}-m_{r}}}{m_{1}^{2}m_{2}\cdots m_{r}}\frac{1}{m_{r}}\sum_{m_{r}\geq n_{1}\geq\cdots\geq m_{p}\geq3}\frac{1}{n_{1}\cdots n_{p}}\right)2^{p+r+1}\qquad(p=d-r-1)\\
 & =\sum_{m_{1}\geq\cdots\geq m_{r}\geq3}\frac{a_{1}^{m_{1}-m_{2}}\cdots a_{r-1}^{m_{r-1}-m_{r}}}{m_{1}^{2}m_{2}\cdots m_{r}}\cdot
 \frac{2^{r+1}}{m_{r}}\prod_{n=3}^{m_{r}}\sum_{c=0}^{\infty}\left(\frac{2}{n}\right)^{c}\\
 & =\sum_{d=r}^{r}a_{d}\left(\sum_{m_{1}\geq\cdots\geq m_{d}\geq3}\frac{a_{1}^{m_{1}-m_{2}}\cdots a_{d-1}^{m_{d-1}-m_{d}}}{m_{1}^{2}m_{2}\cdots m_{d}}\right)2^{d}
 +\sum_{m_{1}\geq\cdots\geq m_{r}\geq3}\frac{a_{1}^{m_{1}-m_{2}}\cdots a_{r-1}^{m_{r-1}-m_{r}}}{m_{1}^{2}m_{2}\cdots m_{r}}2^{r}(m_{r}-2).%\qedhere
\end{align*}
Thus the theorem is proved.
\end{proof}

%%%%%%%%%%%%%%%%%%%%%%%%%%%%%%%%%%%%%%%%%%%%%%%%%%%%%%%%%%%%%%%%%%%%%%%%%%%%%%%%%%%%%%%%%%%%%%%%%%%%%%%
\section{Divergence of left-differential}
In this section, we give a proof of Theorem \ref{main3} (3).
% Theorem \ref{thm:ddd} gives the left-differential divergence of Theorem \ref{main3} (3).
\begin{lem} \label{61}
Fix $r\geq1$. Then we have
\[
\sum_{m_{1}\geq\cdots\geq m_{s}\geq n}\frac{1}{m_{1}^{r+1}m_{2}\cdots m_{s}}\leq\frac{1}{(n-1)\cdots(n-r)r^{s}}
\]
for all $s\geq1$ and $n\geq r+1$. Furthermore, there exists $C_{r}\in\mathbb{R}_{>0}$
such that
\[
\sum_{m_{1}\geq\cdots\geq m_{s}\geq n}\frac{1}{m_{1}^{r+1}m_{2}\cdots m_{s}}
\ge
\frac{1}{(n-1)\cdots(n-r)r^{s}}
-\frac{C_{r}}{(n-1)\cdots(n-r)(n-r-1)(r+1)^{s}}
\]
for all $s\geq1$ and $n\geq r+2$.
\end{lem}
\begin{proof}
The first claim follows from
\begin{align*}
 & \sum_{m_{1}\geq\cdots\geq m_{s}\geq n}\frac{1}{m_{1}^{r+1}m_{2}\cdots m_{s}}\\
 & \leq\sum_{m_{1}\geq\cdots\geq m_{s}\geq n}\frac{1}{m_{1}(m_{1}-1)\cdots(m_{1}-r)m_{2}\cdots m_{s}}\\
 &=\sum_{m_{1}\geq\cdots\geq m_{s}\geq n}
  \frac{1}{r}\left(
   \frac{1}{(m_{1}-1)\cdots(m_{1}-r)}
   -\frac{1}{m_{1}(m_{1}-1)\cdots(m_{1}-r+1)}
  \right)
  \frac{1}{m_{2}\cdots m_{s}}\\
 &=\sum_{m_{2}\geq\cdots\geq m_{s}\geq n}
  \frac{1}{r}
  \cdot\frac{1}{(m_{2}-1)\cdots(m_{2}-r)}
  \cdot\frac{1}{m_{2}\cdots m_{s}}\\
 &=\cdots\\
 &=\frac{1}{(n-1)\cdots(n-r)r^{s}}.
\end{align*}
There exists $C_{r}\in\mathbb{R}_{>0}$ such that
\[
\frac{1}{m^{r+1}}\geq\frac{1}{m(m-1)\cdots(m-r)}-\frac{C_{r}}{m(m-1)\cdots(m-r-1)}
\]
for all $m\geq r+2$. Then, the second claim follows from
\begin{align*}
 & \sum_{m_{1}\geq\cdots\geq m_{s}\geq n}\frac{1}{m_{1}^{r+1}m_{2}\cdots m_{s}}\\
 & \geq\sum_{m_{1}\geq\cdots\geq m_{s}\geq n}\frac{1}{m_{1}(m_{1}-1)\cdots(m_{1}-r)m_{2}\cdots m_{s}}\\
 &\quad -C_{r}\sum_{m_{1}\geq\cdots\geq m_{s}\geq n}\frac{1}{m_{1}(m_{1}-1)\cdots(m_{1}-r-1)m_{2}\cdots m_{s}}\\
 & =\frac{1}{(n-1)\cdots(n-r)r^{s}}-\frac{C_{r}}{(n-1)\cdots(n-r)(n-r-1)(r+1)^{s}}. \qedhere
\end{align*}
\end{proof}

\begin{lem} \label{62}
Fix $r\geq 1$. Then there exists $D_{r}\in\mathbb{R}_{>0}$ such that
\[
\frac{1}{r!r^{s}}-\frac{sD_{r}}{(r+1)^{s}}\leq\sum_{m_{1}\geq\cdots\geq m_{s}\geq r+1}\frac{1}{m_{1}^{r+1}m_{2}\cdots m_{s}}\leq\frac{1}{r!r^{s}}
\]
for all $s\geq 1$.
\end{lem}
\begin{proof}
We have
\begin{align*}
 \sum_{m_{1}\geq\cdots\geq m_{s}\geq r+1}\frac{1}{m_{1}^{r+1}m_{2}\cdots m_{s}}
 & =\sum_{i=0}^{s}\sum_{\substack{m_{1}\geq\cdots\geq m_{i}\geq r+2\\
m_{i+1}=\cdots=m_{s}=r+1
}
}\frac{1}{m_{1}^{r+1}m_{2}\cdots m_{s}}\\
 & \geq\sum_{i=1}^{s}\frac{1}{(r+1)^{s-i}}\sum_{m_{1}\geq\cdots\geq m_{i}\geq r+2}\frac{1}{m_{1}^{r+1}m_{2}\cdots m_{i}},
\end{align*}
and by the second claim of the previous lemma,
\begin{align*}
 & \geq\sum_{i=1}^{s}\frac{1}{(r+1)^{s-i}}\left(\frac{1}{(r+1)!r^{i}}-\frac{C_{r}}{(r+1)!(r+1)^{i}}\right)\\
 & =\frac{1}{r!r^{s}}-\frac{1}{(r+1)!(r+1)^{s-1}}-\frac{sC_{r}}{(r+1)!(r+1)^{s}}
\end{align*}
This proves the lower bound for the inequality, and the upper bound follows from the first claim of the previous lemma.
\end{proof}
\begin{lem} \label{63}
Fix $r\geq 1$. Then there exists $E_{r}\in\mathbb{R}_{>0}$ such that
\begin{align*}
 \frac{s-E_{r}}{r!r^{s}}
 &\leq\sum_{m_{1}\geq\cdots\geq m_{s}\geq r}
 \frac{1}{m_{1}^{r+1}m_{2}\cdots m_{s}}
 -\frac{1}{r^{r+s}}
 \leq\frac{s}{r!r^{s}}
\end{align*}
for all $s\geq 1$.
\end{lem}

\begin{proof}
We have
\begin{align*}
 \sum_{m_{1}\geq\cdots\geq m_{s}\geq r}\frac{1}{m_{1}^{r+1}m_{2}\cdots m_{s}}
 & =\frac{1}{r^{r+s}}+\sum_{i=1}^{s}\sum_{\substack{m_{1}\geq\cdots\geq m_{i}\geq r+1\\
m_{i+1}=\cdots=m_{s}=r
}
}\frac{1}{m_{1}^{r+1}m_{2}\cdots m_{s}}\\
 & =\frac{1}{r^{r+s}}+\sum_{i=1}^{s}\frac{1}{r^{s-i}}\sum_{m_{1}\geq\cdots\geq m_{i}\geq r+1}\frac{1}{m_{1}^{r+1}m_{2}\cdots m_{i}}.
\end{align*}
Here, by the previous lemma,
\[
\frac{1}{r!r^{i}}-\frac{iD_{r}}{(r+1)^{i}}\leq\sum_{m_{1}\geq\cdots\geq m_{i}\geq r+1}\frac{1}{m_{1}^{r+1}m_{2}\cdots m_{i}}\leq\frac{1}{r!r^{i}}.
\]
Furthermore,
\[
\sum_{i=1}^{s}\frac{1}{r^{s-i}}\left(\frac{1}{r!r^{i}}\right)=\frac{s}{r!r^{s}}
\]
and
\begin{align*}
\sum_{i=1}^{s}\frac{1}{r^{s-i}}\left(\frac{iD_{r}}{(r+1)^{i}}\right) & =D_{r}r\left(r^{1-s}+r^{-s}-r(r+1)^{-s}-(s+1)(r+1)^{-s}\right)\\
 & \leq D_{r}r\left(\frac{r+1}{r^{s}}\right).
\end{align*}
Putting $E_r=D_r r(r+1)r!$, we obtain the lemma.
\end{proof}
By Lemmas \ref{61} and \ref{63}, we have
\[
\sum_{m_{1}\geq\cdots\geq m_{s}\geq n}\frac{1}{m_{1}^{r+1}m_{2}\cdots m_{s}}\asymp_{r}\begin{cases}
\frac{1}{(n-1)\cdots(n-r)r^{s}} & n>r\\
\frac{s}{r^{s}} & n=r.
\end{cases}
\]
Thus we have the following:
\begin{lem} \label{eqeq}
Fix $r\geq 1$.
We have
\begin{align*} 
 \sum_{n=r}^{\infty}\sum_{m_{1}\geq\cdots\geq m_{s}\geq n}
 \frac{1}{m_{1}^{r+1}m_{2}\cdots m_{s}}
 \asymp\frac{s}{r^{s}} 
\end{align*} 
for all $s\ge1$.
\end{lem}

\begin{thm} \label{thm:ddd}
Fix $p>0$ and put $z=1-\frac{1}{2^{p}}$. Then for $h=\frac{1}{2^{q}}$ with $q>p$,
we have
\begin{align*}
\frac{Z^{\star}(z)-Z^{\star}(z-h)}{h} & \asymp q.
\end{align*}
Thus, $Z^{\star}$ is not left-differentiable at $1-\frac{1}{2^p}$.
\end{thm}

\begin{proof}
We only consider the case $q>p+1$.
Note that
\[
 z=\sum_{\substack{k=1\\k\neq p}}^{\infty}\frac{1}{2^{k}}
 \quad\text{ and }\quad
 z-h=\sum_{\substack{k=1\\k\neq p,q}}^{\infty}\frac{1}{2^{k}}.
\]
By definition, we have
\[
Z^{\star}(z)=\sum_{\substack{m_{1}\geq m_{2}\geq\cdots\geq1\\
m_{p}=m_{p+1}
}
}\frac{1}{m_{1}^{2}m_{2}m_{3}\cdots}
\]
and
\[
Z^{\star}(z-h)=\sum_{\substack{m_{1}\geq m_{2}\geq\cdots\geq1\\
m_{p}=m_{p+1},m_{q}=m_{q+1}
}
}\frac{1}{m_{1}^{2}m_{2}m_{3}\cdots}.
\]
Then we have
\begin{align*}
 Z^{\star}(z)-Z^{\star}(z-h) 
 &=\sum_{\substack{m_{1}\geq m_{2}\geq\cdots\geq1\\
m_{p}=m_{p+1},m_{q}>m_{q+1}
}
}\frac{1}{m_{1}^{2}m_{2}m_{3}\cdots}\\
& \asymp\sum_{\substack{m_{1}\geq m_{2}\geq\cdots\geq1\\
m_{p}=m_{p+1},m_{q}>m_{q+1}
}
}\frac{1}{m_{1}(m_{1}-1)m_{2}m_{3}\cdots}\\
 &=\sum_{\substack{m_{p}\geq m_{p+1}\geq\cdots\geq1\\
m_{p}=m_{p+1},m_{q}>m_{q+1}
}
}\frac{1}{m_{p}(m_{p}-1)m_{p+1}m_{p+2}\cdots}\\
 &=\sum_{\substack{m_{p+1}\geq\cdots\geq1\\
m_{q}>m_{q+1}
}
}\frac{1}{m_{p+1}^{2}(m_{p+1}-1)m_{p+2}m_{p+3}\cdots}.
\end{align*}
Similar to the proof of Lemma \ref{lem:linf}, we have
\begin{align*}
 &Z^{\star}(z)-Z^{\star}(z-h) \\
 &\asymp\sum_{m_{p+1}\geq\cdots\geq m_{q}\geq2}\frac{1}{m_{p+1}^{2}(m_{p+1}-1)m_{p+2}m_{p+3}\cdots m_{q}}\sum_{m_{q}>m_{q+1}\geq m_{q+2}\cdots}\frac{1}{m_{q+1}m_{q+2}\cdots}\\
% &=\sum_{m_{p+1}\geq\cdots\geq m_{q}\geq2}\frac{1}{m_{p+1}^{2}(m_{p+1}-1)m_{p+2}m_{p+3}\cdots m_{q}}\prod_{m=2}^{m_{q}-1}\sum_{s=0}^{\infty}\frac{1}{m^{s}}\\
 &=\sum_{m_{p+1}\geq\cdots\geq m_{q}\geq2}\frac{1}{m_{p+1}^{2}(m_{p+1}-1)m_{p+2}m_{p+3}\cdots m_{q}}\prod_{m=2}^{m_{q}-1}\frac{m}{m-1}\\
 &=\sum_{m_{p+1}\geq\cdots\geq m_{q}\geq2}\frac{m_{q}-1}{m_{p+1}^{2}(m_{p+1}-1)m_{p+2}m_{p+3}\cdots m_{q}}\\
 &\asymp\sum_{m_{p+1}\geq\cdots\geq m_{q}\geq2}\frac{1}{m_{p+1}^{3}m_{p+2}m_{p+3}\cdots m_{q-1}}.
\end{align*}
Hence we get
\begin{align*}
 Z^{\star}(z)-Z^{\star}(z-h) 
 &\asymp\sum_{n=2}^{\infty}\sum_{m_{p+1}\geq\cdots\geq m_{q-1}\geq n}\frac{1}{m_{p+1}^{3}m_{p+2}m_{p+3}\cdots m_{q-1}}\\
 &=\sum_{n=2}^{\infty}\sum_{m_{1}\geq\cdots\geq m_{s}\geq n}\frac{1}{m_{1}^{3}m_{2}m_{3}\cdots m_{s}},
\end{align*}
where $s=q-p-1>0.$
By Lemma \ref{eqeq}, we find
\begin{align*}
 Z^{\star}(z)-Z^{\star}(z-h) 
 &\asymp \frac{s}{2^{s}} \asymp qh,
\end{align*}
which completes the proof.
\end{proof}

%%%%%%%%%%%%%%%%%%%%%%%%%%%%%%%%%%%%%%%%%%%%%%%%%%%%%%%%%%%%%%%%%%%%%%%%%%%%%%%%%%%%%%%%%%%%%%%%%%%%%%%
% Acknowledgements %
%%%%%%%%%%%%%%%%%%%%%%%%%%%%%%%%%%%%%%%%%%%%%%%%%%%%%%%%%%%%%%%%%%%%%%%%%%%%%%%%%%%%%%%%%%%%%%%%%%%%%%%
\section*{Acknowledgements}
This work was supported by JSPS KAKENHI Grant Numbers JP18K13392, JP19K14511, JP22K03244, and JP22K13897.

%%%%%%%%%%%%%%%%%%%%%%%%%%%%%%%%%%%%%%%%%%%%%%%%%%%%%%%%%%%%%%%%%%%%%%%%%%%%%%%%%%%%%%%%%%%%%%%%%%%%%%%

\end{document}